\documentclass[11pt]{article}%
\usepackage{amssymb}
\usepackage{amsmath}
\usepackage{amsfonts}
\usepackage{graphicx}%
\newtheorem{theorem}{Theorem}[section]

\newtheorem{corollary}[theorem]{Corollary}

\newtheorem{definition}[theorem]{Definition}
\newtheorem{example}[theorem]{Example}

\newtheorem{lemma}[theorem]{Lemma}

\newtheorem{proposition}[theorem]{Proposition}
\newtheorem{remark}[theorem]{Remark}

\newenvironment{proof}[1][Proof]{\noindent\textbf{#1.} }{\ \rule{0.5em}{0.5em}}
\begin{document}
\author{Yulia Kempner\\Department of Computer Science\\Holon Institute of Technology, Israel\\yuliak@hit.ac.il
\and Vadim E. Levit\\Department of Computer Science \\Ariel University, Israel\\levitv@ariel.ac.il}
\date{}
\title{Cospanning characterizations of antimatroids and convex geometries }
\maketitle

\begin{abstract}

Given a finite set $E$ and an operator $\sigma:2^{E}\longrightarrow2^{E}$, two
sets $X,Y\subseteq E$ are \textit{cospanning} if $\sigma\left(  X\right)
=\sigma\left(  Y\right)  $. Corresponding \textit{cospanning equivalence
relations} were investigated for greedoids in much detail (Korte, Lovasz,
Schrader; 1991). For instance, these relations determine greedoids uniquely.
In fact, the feasible sets of a greedoid are exactly the inclusion-wise minimal sets of the equivalence classes.

In this research, we show that feasible sets of convex geometries
are the inclusion-wise maximal sets of the equivalence classes of the
corresponding closure operator. Same as greedoids, convex geometries are
uniquely defined by the corresponding cospanning relations. For each closure
operator $\sigma$, an element $x\in X$ is \textit{an extreme point} of $X$ if
$x\notin\sigma(X-x)$. The set of extreme points of $X$ is denoted by $ex(X)$.
We prove, that if $\sigma$ has the anti-exchange property, then for every set
$X$ its equivalence class $[X]_{\sigma}$ is the interval
$[ex(X),\sigma(X)]$. It results in the one-to-one correspondence between the
cospanning partitions of an antimatroid and its complementary convex geometry.

The obtained results based on the connection between violator spaces, greedoids and antimatroids.
A notion of \textit{violator space} was introduced in (G\"{a}rtner,
Matou\v{s}ek, R\"{u}st, \v{S}kovro\v{n}by; 2008) as a combinatorial framework
that encompasses linear programming and other geometric optimization problems.

Violator spaces are defined by violator operators. We have proved  that a violator operator may be
defined by a weak version of a closure operator. In the paper we prove that violator
operators generalize rank closure operators of greedoids. Further, we
introduced \textit{co-violator spaces} based on contracting operators known
also as choice functions. Cospanning characterization of these combinatorial structures allows us not only to give the new characterization of antimatroids and convex geometries, but also to obtain the new properties of closure operators, extreme point operators and
their interconnections.


\end{abstract}

\section{Preliminaries}

While convex geometries are usually defined as a closure space with anti-exchange property, antimatroids - the families of sets complementary to convex sets, are described as accessible set systems closed under union. At the same time matroids can be defined by two ways - by a closure operator with Steinitz-MacLine exchange property and as  families of independent sets. Greedoids, that are a common generalization of matroids and antimatroids, are defined as set systems with augmentation property, but they may be described by some "closure operator" as well.

This work is an attempt to describe all these structures in identical way.
Each set operator determines the partition of sets to equivalence classes with equal value of the operator. Let we have some set operator $\alpha$. Following \cite{Greedoids} we call two sets $X,Y$ to be \textit{cospanning} if $\alpha(X)=\alpha(Y)$. Thus each set operator generates the cospanning equivalence relation on sets. The cospanning relation associated with a closure operator of greedoids was introduced and investigated in \cite{Greedoids}, where it was proved that the relation determines the greedoid uniquely. In fact, the feasible sets of a greedoid are exactly the minimal  (by inclusion) sets of the equivalence classes, i.e., the sets that are not cospanning with any proper subset.
We extend the approach to other combinatorial structures. Thus convex geometries, being families of closed sets, are the maximal (by inclusion) sets of the equivalence classes of the closure operator, and matroids are the minimal sets.

 We begin with a definition of  a violator operator  which, as we prove,  is a generalization of a closure operator and a rank closure operator of greedoids. The rest of the paper is organized as follows. In Section 2, we give a brief introduction to violator spaces focussing on violator spaces with unique basis. At the end of the section we characterize the cospanning relation with regards to violator spaces and describe the equivalence classes of the relation. Section 3 is devoted to closure spaces with a focus on convex geometries, and in Section 4 we investigate the cospanning relation of antimatroids and matroids and prove one-to-one correspondence between cospanning partition of antimatroids and convex geometries.

\section{ Violator spaces}

Violator spaces are arisen
as generalization of Linear Programming problems. LP-type problems have been
introduced and analyzed by Matou\v{s}ek, Sharir and Welzl \cite{MSW, SW} as a
combinatorial framework that encompasses linear programming and other
geometric optimization problems. Further, Matou\v{s}ek et al. \cite{VS} define a simpler
framework: violator spaces, which constitute a proper generalization of
LP-type problems.  Originally, violator spaces
were defined for a set of constraints $E$, where each subset of constraints
$G\subseteq E$ was associated with $\nu(G)$ - the set of all constraints
violating $G$.

The classic example of an LP-type problem is the problem of computing the
smallest enclosing ball of a finite set of points in $R^{d}$. Here $E$ is a
set of points in $R^{d}$, and the violated constraints of some subset of the
points $G$ are exactly the points lying outside the smallest enclosing ball of
$G$.

\begin{definition}
\cite{VS} A \textit{violator space} is a pair $(E,\nu)$, where $E$ is a finite
set and $\nu$ is a mapping $2^{E}\rightarrow2^{E}$ such that for all subsets
$X,Y\subseteq E$ the following properties are satisfied:

\textbf{V11}: $X \cap\nu(X) = \emptyset$ (consistency),

\textbf{V22}: $(X\subseteq Y$ and $Y \cap\nu(X) = \emptyset) \Rightarrow
\nu(X)=\nu(Y)$ (locality).
\end{definition}

Let $(E,\nu)$ be a violator space. Define $\varphi(X)=E-\nu(X)$.
The operator $\varphi$ satisfies extensivity ($X\subseteq\varphi(X)$), self-convexity $(X\subseteq Y\subseteq\varphi(X))\Rightarrow\varphi (X)=\varphi(Y)$  and idempotence ($\varphi(\varphi(X))=\varphi(X)$)(\cite{KempnerLevit}).

In what follows, if $(E,\nu)$ is a violator space and $\varphi(X)=E-\nu(X)$,
then $(E, \varphi)$ will be called a violator space and $ \varphi$ - a violator operator as well.

\begin{definition} (\cite{KempnerLevit})
A \textit{violator space} is a pair $(E,\varphi)$, where $E$ is a finite set
and $\varphi$ is an operator $2^{E}\rightarrow2^{E}$ such that for all subsets
$X,Y\subseteq E$ the following properties are satisfied:

\textbf{V1}: $X\subseteq\varphi(X)$ (extensivity),

\textbf{V2}: $(X\subseteq Y\subseteq\varphi(X))\Rightarrow\varphi
(X)=\varphi(Y)$ (self-convexity).
\end{definition}

\begin{lemma}\label{lemma2}
If $\varphi$ satisfies extensivity and self-convexity, then for every $A,B$
$B \subseteq \varphi(A) \Leftrightarrow \varphi(A)=\varphi(A \cup B)$.
\end{lemma}

\begin{proof}
"\textit{If}": Let $\varphi(A)=\varphi(A \cup B)$. Since extensivity implies
$B \subseteq A \cup B \subseteq \varphi(A \cup B)$, we conclude with $B \subseteq \varphi(A)$.

"\textit{Only if}": If $B \subseteq \varphi(A)$, then $A\subseteq A\cup B \subseteq \varphi(A)$.
Hence, by self-convexity, $\varphi(A)=\varphi(A\cup B)$.
\end{proof}

\begin{corollary}
\label{Lemma} If $\varphi$ satisfies extensivity and self-convexity, then for every $X \subseteq E$ and $x \in E$
\[
(x\in \varphi(X))\Leftrightarrow(\varphi(X)=\varphi(X\cup x)).
\]
\end{corollary}

\begin{proposition}
\label{SC_G}
For each extensive operator $\varphi$
self-convexity is equivalent to the following property:
\textbf{VV2}: $(X\subseteq \varphi(Y) \wedge Y\subseteq \varphi(X))\Rightarrow\varphi
(X)=\varphi(Y)$.
\end{proposition}

\begin{proof}
1. \textbf{V2} $\Rightarrow$ \textbf{VV2}: From Lemma \ref{lemma2}  $(X\subseteq \varphi(Y) \wedge Y\subseteq \varphi(X))\Rightarrow \varphi(Y)= \varphi(X \cup Y)= \varphi(X)$.  Hence, $\varphi(X)=\varphi(Y)$.

2. \textbf{VV2} $\Rightarrow$ \textbf{V2}: $(X\subseteq Y\subseteq\varphi(X))\Rightarrow X\subseteq \varphi(Y) \wedge Y\subseteq \varphi(X)$ (from \textbf{V1}). Then \textbf{VV2} implies $\varphi
(X)=\varphi(Y)$.
\end{proof}

\begin{lemma}(\cite{KempnerLevit})
\label{un} Let $(E,\varphi)$ be a violator space. Then
\begin{equation*}
\varphi(X)=\varphi(Y)\Rightarrow\varphi(X\cup Y)=\varphi(X)=\varphi(Y)
\label{Union}%
\end{equation*}
and
\begin{equation*}
(X \subseteq Y \subseteq Z) \wedge (\varphi(X)=\varphi(Z)) \Rightarrow \varphi(X)=\varphi(Y)=\varphi(Z)
\label{Convexity}%
\end{equation*}
for every $X,Y,Z\subseteq E$.
\end{lemma}

Since the second property deals with
all sets lying between two given sets, following \cite{Monjardet} we call the property  \textit{convexity}.

\subsection{Uniquely generated violator spaces}

Let $(E,\alpha)$ be an arbitrary space with the operator $\alpha:2^{E}%
\rightarrow2^{E}$. $B\subseteq E$ is a \textit{generator} of
$X\subseteq E$ if $\alpha(B)=\alpha(X)$. For $X\subseteq E$, a \textit{basis}
(minimal generator) of $X$ is a inclusion-minimal set $B\subseteq E$ (not
necessarily included in $X$) with $\alpha(B)=\alpha(X)$. A space $(E,\alpha)$
is \textit{uniquely generated} if every set $X\subseteq E$ has a unique basis.

\begin{proposition}\cite{KempnerLevit}
\label{UQP} A violator space $(E,\varphi)$ is uniquely generated if and only if
for every $X,Y\subseteq E$
\begin{equation}
\varphi(X)=\varphi(Y)\Rightarrow\varphi(X\cap Y)=\varphi(X)=\varphi(Y) \label{UQ}%
\end{equation}

\end{proposition}

We can rewrite the property (\ref{UQ}) as follows: for every set $X\subseteq
E$ of a uniquely generated violator space $(E,\varphi)$, the basis $B$ of $X$ is
the intersection of all generators of $X$:
\begin{equation}
B=\bigcap\{Y\subseteq E:\varphi(Y)=\varphi(X)\}. \label{UQI}%
\end{equation}

One of the known examples of a not uniquely generated violator space is the
violator space associated with the smallest enclosing ball problem. A basis of
a set of points is a minimal subset with the same enclosing ball. In
particular, all points of the basis are located on the ball's boundary. For
$R^{2}$ the set $X$ of the four corners of a square has two bases: the two
pairs of diagonally opposite points. Moreover, one of these pairs is a basis
of the second pair. Thus the equality (\ref{UQI}) does not hold.

It is known that a closure operator $\tau$ is uniquely generated if and only if it
satisfies the \textit{anti-exchange property} \cite{Edelman, Greedoids, CL}:
\[
p,q\notin\tau(X)\wedge p\in\tau(X\cup{q})\Rightarrow q\notin\tau(X\cup{p}).
\]
We extended this characterization to violator spaces.

\begin{theorem}\cite{KempnerLevit}
\label{TH} Let $(E,\varphi)$ be a violator space. Then $(E,\varphi)$ is uniquely
generated if and only if the operator $\varphi$ satisfies the anti-exchange property.
\end{theorem}

For each  arbitrary space $(E,\alpha)$ with the operator $\alpha:2^{E}\rightarrow2^{E}$, an element $x$ of a subset $X\subseteq E$ is \textit{an extreme point} of
$X$ if $x\notin\alpha(X-x)$. The set of extreme points of $X$ is denoted by
$ex(X)$.

\begin{proposition}\cite{KempnerLevit}
\label{expoint}
Let $(E,\varphi)$ be a violator space. Then $x\in ex(X)\ $if and only if
$\varphi(X) \neq \varphi(X-x)$.
\end{proposition}

\begin{proposition}
\cite{KempnerLevit} \label{exp} Let $(E,\varphi)$ be a violator space. Then
\begin{equation*}
ex(X)=\bigcap\{B\subseteq X:\varphi(B)=\varphi(X)\}.
\end{equation*}
\end{proposition}

\begin{proposition}\cite{KempnerLevit}
\label{exp-vs}
Let $(E,\varphi)$ be a violator space. Then $ex(\varphi(X))\subseteq ex(X)$.
\end{proposition}

\subsection{Co-violator spaces}

\begin{definition}\cite{Co-V}
A \textit{co-violator space} is a pair $(E,c)$, where $E$ is a finite set
and $c$ is an operator $2^{E}\rightarrow2^{E}$ such that for all subsets
$X,Y\subseteq E$ the following properties are satisfied:

\textbf{CV1}: $c(X)\subseteq X$,

\textbf{CV2}: $(c(X)\subseteq Y\subseteq X)\Rightarrow c(X)=c(Y)$.
\end{definition}

Operators which satisfy the property \textbf{CV1} are called contracting operators.

In social sciences, contracting operators are called choice functions, usually adding a requirement that $c(X)\neq\emptyset$ for every $X\neq\emptyset$.  The property \textbf{CV2} is called the \textit{outcast property} or the \textit{Aizerman property} (\cite{Monjardet}).

The properties of co-violator spaces match to the corresponding
("mirrored") properties of violator spaces. Thus, every co-violator operator $c$ is idempotent. Indeed, since $c$ is contracting  $c(X) \subseteq c(X)\subseteq X$. Then, \textbf{CV2}
implies $c(c(X))=c(X)$.

Lemma \ref{un} is converted to the following.
\begin{lemma}\cite{Co-V}
\label{co-un} Let $(E,c)$ be a co-violator space. Then
\begin{equation*}
c(X)=c(Y)\Rightarrow c(X\cap Y)=c(X)=c(Y)
\label{Intersection}%
\end{equation*}
and
\begin{equation*}
(X \subseteq Y \subseteq Z) \wedge (c(X)=c(Z)) \Rightarrow c(X)=c(Y)=c(Z)
\label{Co-Convexity}%
\end{equation*}
for every $X,Y,Z\subseteq E$.
\end{lemma}

  Given an extensive operator $\varphi:2^{E}\rightarrow2^{E}$, one can get a contracting operator $c$: $c(X)=E-\varphi(E-X)$ or $\overline{c(X)}=\varphi(\overline{X})$. In topology this construction is known as \textit{interior operator} dual to a closure operator.

\begin{proposition}\cite{Co-V}
\label{co-operator}
$(E,\varphi)$ is a violator space if and only if $(E,c)$ is a co-violator space, where $c(X)=\overline{\varphi(\overline{X})}$.
\end{proposition}

\subsection{Cospanning relations of violator spaces}

Let $E=\left\{ {x_{1},x_{2},...,x_{d}}\right\} $. The graph $H(E)$ is
defined as follows. The vertices are the finite subsets of $E$, two vertices
$A$ and $B$ are adjacent if and only if they differ in exactly one element.
Actually, $H(E)$ is \textit{the hypercube} on $E$ of dimension $d$, since
the hypercube is known to be equivalently considered as the graph on the
Boolean space $\{0,1\}^{d}$ in which two vertices form an edge if and only
if they differ in exactly one position.

Let $(E,\varphi)$ be a violator space. The two sets $X$ and $Y$ are \textit{%
equivalent} (or \textit{cospanning}) if $\varphi(X)=\varphi(Y)$. In what
follows, $\mathcal{P}$ denotes a partition of $H(E)$ ( or $2^{E}$) into
equivalence classes with regard to this relation, and $[A]_{\varphi}:=\{X%
\subseteq E:\varphi(X)=\varphi(A)\}$.

The following theorem characterizes the cospanning relation of violator spaces.
\begin{theorem}\cite{Co-V}
\label{T_rel}
Let $E$ be a finite set and $R \subseteq 2^{E} \times 2^{E}$ be an equivalence relation on $2^{E}$. Then $R$ is the cospanning relation of a violator space if and only if the following properties hold for every $X,Y,Z \subseteq E$:

\textbf{R1}: if $(X,Y) \in R$, then $(X,X \cup Y) \in R$

\textbf{R2}: if $X \subseteq Y \subseteq Z$ and $(X,Z) \in R$, then $(X,Y) \in R$.

\end {theorem}

Thus each equivalence class of the cospanning relation of a violator space is closed under union (\textbf{R1}) and convex (\textbf{R2}).

Similarly,
\begin{theorem}\cite{Co-V}
\label{cv_rel}
Let $E$ be a finite set and $R \subseteq 2^{E} \times 2^{E}$ be an equivalence relation on $2^{E}$. Then $R$ is the cospanning relation of a co-violator space if and only if the following properties hold for every $X,Y,Z \subseteq E$:

\textbf{R3}: if $(X,Y) \in R$, then $(X,X \cap Y) \in R$

\textbf{R2}: if $X \subseteq Y \subseteq Z$ and $(X,Z) \in R$, then $(X,Y) \in R$.

\end {theorem}

Consider now a violator operator $\varphi$ and a co-violator operator $c(X)=\overline{\varphi(\overline{X})}$.

\begin{proposition}\cite{Co-V}
\label{co-co}
There is a one-to-one correspondence between an equivalence class $[X]_{\varphi}$ of $X$ of the cospanning relation associated with a violator operator $\varphi$ and an equivalence class $[\overline{X}]_{c}$ w.r.t. a co-violator operator $c$, i.e., $A \in [X]_{\varphi}$ if and only if $\overline{A} \in [\overline{X}]_{c}$.
 \end{proposition}

 An uniquely generated violator space defines a cospanning relation with additional property \textbf{R3} (see Proposition \ref{UQP}).

So every uniquely generated violator space is a co-violator space as well. Each equivalence class of the cospanning relation of a uniquely generated violator space has an unique minimal element and an unique maximal element. More precisely,
for the sets $A\subseteq B\subseteq E$, let $[A,B]:=\{C\subseteq E:A\subseteq
C\subseteq B\}$. One calls any $[A,B]$ an \textit{interval}. Then each equivalence class of an uniquely generated violator space is an interval. We call a partition of $H(E)$ into disjoint intervals a
\textit{hypercube partition}. The following Theorem follows immediately from Theorem \ref{T_rel} and Proposition \ref{UQP}.

\begin{theorem} (\cite{Clarkson})
(i) If $(E,\varphi)$ is a uniquely generated violator space, then $\mathcal{P}$ is a hypercube partition of $H(E)$.

(ii) Every hypercube partition is the partition $\mathcal{P}$ of $H(E)$ into
equivalence classes of a uniquely generated violator space.
\end{theorem}

More specifically \cite{KempnerLevit}, $[A]_{\varphi}=[ex(A),\varphi(A)]$ for every set $A \subseteq E$.

It is interesting notice that for a cospanning relation $R$ of a violator space, i.e., for an equivalence relation satisfying \textbf{R1} and \textbf{R2}, the property \textbf{R3} is equivalent to the anti-exchange property. First, rewrite the anti-exchange property in terms of a cospanning relation $R$. From Corollary \ref{Lemma} it follows

 $p \notin \varphi(X) \Leftrightarrow \varphi(X)\neq \varphi(X\cup p) \Leftrightarrow (X,X\cup p) \notin R$.

Then the anti-exchange property looks as

\textbf{R33}: if $(X,X \cup p) \notin R$ , $(X,X \cup q) \notin R$, and $(X \cup p,X \cup p \cup q) \in R$, then $(X \cup q,X \cup p \cup q) \notin R$.

\begin{proposition}
An equivalence relation $R$ satisfying the property \textbf{R1},\textbf{R2} and \textbf{R3} coincides with an
equivalence relation satisfying the property \textbf{R1},\textbf{R2} and \textbf{R33}.
\end{proposition}

\begin{proof}
Let $R$ be an equivalence relation satisfying the property \textbf{R1} and \textbf{R2}.

1. Prove, that if $R$ satisfies \textbf{R3}, then \textbf{R33} holds. Suppose there are $X \in E, p,q \notin X$ for which \textbf{R33} does not hold, i.e., $(X \cup q,X \cup p \cup q) \in R$. Since $R$ is an
equivalence relation we have $(X \cup q,X \cup p) \in R$, but from \textbf{R3} it follows that $(X,X \cup p) \in R$. Contradiction.

2. Prove, that if $R$ satisfies \textbf{R33}, then \textbf{R3} holds. \textbf{R1} implies uniqueness of maximal element, so consider some equivalence class  with unique maximal element $X$. Let $B$ be a minimal element of $[X]_{R}$. To prove that $B$ is unique
minimal element enough to prove  that $B \subseteq Y$ for each $Y \in [X]_{R}$.
Suppose for the sake of contradiction that there are $p\in B$ and $p\notin
Y$. Since $B-p \subseteq B\subseteq X$ and
$Y\subseteq X$, we have $Y\subseteq(B-p)\cup
Y\subseteq X$. Now, \textbf{R2} implies that $((B-p)\cup
Y,X) \in R$. Since $B$ is a minimal generator of $X$, we obtain that
$((B-p),X)\notin R$. Thus one can see that there is a minimal set
$C$ such that $\emptyset \subset C\subseteq Y$ and $((B-p)\cup
C),X)\in R$.

Consider some element $q\in C$, and let $Z=(B-p)\cup(C-q)$. Based on
minimality of $C$ it follows that $(Z,X)\notin R$. Note that
$(Z\cup p,X)\in R$, which follows from $B \subseteq Z\cup
p\subseteq X$ and \textbf{R2}. Thus, $(Z,Z\cup p)\notin R$.
Similarly, we
obtain $(Z,Z\cup q)\notin R$, since $(Z\cup q)=(B-p)\cup
C$ and so $(Z\cup q,X)\in R$. Now $(Z\cup
p \cup q,X)\in R$, that means $(Z \cup p,Z \cup p \cup q) \in R$, and $(Z \cup q,Z \cup p \cup q) \in R$, contradicting the \textbf{R33}. Consequently, $B \subseteq Y$ for each $Y \in [X]_{R}$. Hence  \textbf{R3} holds.
\end{proof}

Consider now an uniquely generated violator space $(E,\varphi)$ and operator $ex$. Since each equivalence class $[A]_{\varphi}$ w.r.t. operator $\varphi$  is an interval $[ex(A),\varphi(A)]$, we can see that  for each $X \in [ex(A),\varphi(A)]$ not only $\varphi(X)=\varphi(A)$, but $ex(X)=ex(A)$ as well, and since $\mathcal{P}$ is a hypercube partition of $H(E)$ we conclude  with  $[X]_{\varphi}=[X]_{ex}$. So the cospanning partition (quotient set) associated with an operator $\varphi$ coincides with the cospanning partition associated with a contracting operator $ex$.
Since $ex(X)$ is a minimal element of $[X]$ we immediately obtain
\begin{proposition}\cite{Co-V}
If $(E,\varphi)$ is a uniquely generated violator space, then
 operator $ex$  satisfies the following properties:

 \textbf{X1}: $ex(ex(X))=ex(X)$

\textbf{X2}: $ex(X)=ex(Y)\Rightarrow ex(X\cup Y)=ex(X)=ex(Y)$

\textbf{X3}:$(X \subseteq Y \subseteq Z) \wedge (ex(X)=ex(Z)) \Rightarrow ex(X)=ex(Y)=ex(Z)$

\textbf{X4}: $ex(X)=ex(Y)\Rightarrow ex(X\cap Y)=ex(X)=ex(Y)$
\end{proposition}

\begin{proposition}\cite{Co-V}
\label{outcast_u}
Let $(E,\varphi)$ be a violator space. The following assertions are equivalent:

(i) $(E,\varphi)$ is uniquely generated

(ii) \textbf{X5}: $(ex(X) \subseteq Y \subseteq X) \Rightarrow ex(X)=ex(Y)$ (the outcast property)

(iii) \textbf{X6}: $\varphi(ex(X))=\varphi(X)$

(iv) \textbf{X7}: $ex(\varphi(X))=ex(X)$
\end{proposition}

\section{ Closure spaces}

\begin{definition}
Let $E$ be a finite set. $\tau:2^{E}\rightarrow2^{E}$ is a \textit{closure
operator on }$E$ if for all subsets $X,Y\subseteq E$ the following
properties are satisfied:

\textbf{C1}: $X \subseteq\tau(X)$ (extensivity)

\textbf{C2}: $X \subseteq Y\Rightarrow\tau(X)\subseteq\tau(Y)$ (isotonicity)

\textbf{C3}: $\tau(\tau(X))=\tau(X)$ (idempotence).
\end{definition}

$(E,\tau)$ is a \textit{closure space} if $\tau$ is a closure operator. A set
$A\subseteq E$ is \textit{closed }if $A=\tau(A)$. Clearly, the family of
closed sets $K=\{A\in E:A=\tau(A)\}$ is closed under intersection.
Conversely, any set system $(E,K)$ closed under intersection is a family of
closed sets of the closure operator
\begin{equation}
\label{tau}
\tau_{K}(X)=
{\displaystyle\bigcap}
\{A\in K:X\subseteq A\}.
\end{equation}

In a Euclidean space, a set is convex if it contains the line segment between
any two of its points. It is easy to see that the family of convex sets is
closed under intersection. In fact, the family of convex sets coincides with
the family of closed sets defined by a convex hull operator.

Convex geometries were invented by Edelman and Jamison in 1985 as proper
combinatorial abstractions of convexity \cite{Edelman}. There are various ways
to characterize finite convex geometries. One of them defines convex sets
using anti-exchange closure operators.

A closure space $(E,\tau)$ is a \textit{convex geometry} if it satisfies the
anti-exchange property.
The convex hull operator on Euclidean space is a classic example of a closure
operator with the anti-exchange property.

Consider relations between closure operators and self-convex operators.

\begin{lemma}(\cite{KempnerLevit})
\label{lem2} Isotonicity and idempotence imply self-convexity.
\end{lemma}

So we conclude with

\begin{theorem}(\cite{KempnerLevit})
\label{cltov} Every closure space is a violator space.
\end{theorem}

Let $R$ be a cospanning relation of a closure space. Then every equivalence class of the cospanning relation  is closed under union (\textbf{R1}) and convex (\textbf{R2}). Easy to see that all closed sets are the inclusion-maximal sets of the equivalence classes. First note that idempotence implies $(X, \tau(X)) \in R$. Since $X \subseteq \tau(X)$ the family of closed sets $K$ coincides with the family of  maximal sets of  equivalence classes. Moreover, the cospanning partition of a closure space satisfies the \textit{accessibility property:}

\begin{proposition} \label{Augm_CS}
$\forall X \in K: (X-x,X) \notin R \Leftrightarrow X-x \in K$.
\end{proposition}

\begin{proof}
1. Let  $(X-x,X) \notin R$, and so $\tau(X-x) \neq \tau(X)$. From extensivity and isotonicity of $\tau$ it follows $X-x \subseteq \tau(X-x) \subset \tau(X)=X$. Hence $\tau(X-x) = X-x$, that means $X-x \in K$.

2. If $X-x \in K$, then $\tau(X-x) = X-x$, that means $(X-x,X) \notin R$.
\end{proof}

Since $\tau(X-x) \neq \tau(X)$ is equivalent to $x \in ex(X)$ (see Proposition \ref{expoint}), Proposition \ref{Augm_CS} may be rewritten as follows

\begin{equation}
\label{Augm_P}
\forall X \in K: x \in ex(X) \Leftrightarrow X-x \in K.
\end{equation}

This property is mentioned in \cite{Greedoids,CL} for convex geometries.

If $ex$ is a non-empty choice operator, i.e., $ex(X)\neq\emptyset$ for every $X\neq\emptyset$,  that a closure system is an accessible system  in which every nonempty closed set $X$ contains an element $x$ such that $X-x$ is closed. So in follows we call the property (\ref{Augm_P})  an accessibility property as well.

%
%
%

Consider now a special class of closure spaces - convex geometries. Since convex geometries are uniquely generated closure spaces $(E,\tau)$(\cite{Greedoids,CL}), they determine a hypercube partition $\mathcal{P}$ into equivalence classes
$[A]_{\tau}:=\{X\subseteq E:\tau(X)=\tau(A)\} = [ex(A),\tau(A)]$.
The family of closed sets ($X=\tau(X)$), just being a convex geometry, is exactly the family of maximal sets of the intervals.
In addition the hypercube partition $\mathcal{P}$ has an accessibility property that for cospanning partitions looks as follows.

\begin{equation}
\forall [A,B]\in \mathcal{P} \text{ and } x \in A \text{ there is } C \subseteq E \text{ such that } [C,B-x] \in \mathcal{P}\label{Eq_CL}%
\end{equation}

Let we have some hypercube partition $\mathcal{P^{'}}$ with accessibility property (\ref{Eq_CL}). Denote $\mathcal{N}=\{B:[A,B]\in \mathcal{P^{'}}\}$.
\begin{lemma}
\label{lem_cg}
If $[A,B] \in \mathcal{P^{'}}$, and $C \subseteq B$, but $C \notin [A,B]$, then there exists $x \in B-C$ such that $B-x \in \mathcal{N}$.
\end{lemma}

\begin{proof}
 Since $C \notin [A,B]$, there is $x \in A$ such that $x \notin C$, so $x \in B-C$. Then the property (\ref{Eq_CL}) implies $B-x \in \mathcal{N}$.
\end{proof}

Since the maximal element in each interval is unique we obtain, that the family $\mathcal{N}$ satisfies the \textit{chain property}: for all $X,Y \in \mathcal{N}$, and $X\subset Y$, there exists a chain
$X=X_{0}\subset X_{1}\subset...\subset X_{k}=Y$ such that
$X_{i}=X_{i-1}\cup x_{i}$ and $X_{i}\in\mathcal{N}$ for $0\leq i\leq
k$.

\begin{theorem}
(i) If $(E,\tau)$ is a convex geometry, then equivalence classes of the cospanning relation associated with $\tau$ form a hypercube partition $\mathcal{P^{'}}$ satisfying the accessibility property.

(ii) Every hypercube partition $\mathcal{P^{'}}$ satisfying accessibility property (\ref{Eq_CL}) is the partition  of $H(E)$ into equivalence classes of the cospanning relation of a convex geometry.
\end{theorem}

\begin{proof}
It remains to prove (ii). For each
$X\subseteq E$ there is only one interval $[A,B]$ containing $X$. Then for
every set $X$, we define $\alpha(X)=B$. To prove that $(E,\alpha)$ is a
convex geometry, it is enough
to show that $(E,\alpha)$ is a uniquely generated closure space. To demonstrate it, one has to
check that $\alpha$ satisfies extensivity, isotonicity and idempotence.

Extensivity follows from $X\subseteq B=\alpha(X)$. Idempotence is obviously.
Let $X \subseteq Y$. Prove $\alpha(X)\subseteq \alpha(Y)$. If $X$ and $Y$ belong to the same interval $[A,B]$, then $\alpha(X)=\alpha(Y)=B$. Let $X \in [C,D]$ and $Y \in [A,B]$, then $X \subseteq B$. Hence Lemma \ref{lem_cg} implies that there is a chain $B \supseteq B_{1}=B-x_{1} \supseteq B_{2}=B_{1}-x_{2}\supseteq...\supseteq B_{k}$,  where all elements of the chain $B_{i} \in \mathcal{N}$ and $X \subseteq B_{i}$. The chain ends with $X \in [A_{k},B_{k}]$, i.e., $B_{k}=D=\alpha(X)$. So $\alpha(X) \subseteq B=\alpha(Y)$. Thus the operator $\alpha$ is a closure operator.

To prove unique generation notice that $\alpha(X)=\alpha(Y)$
means $X,Y\in [A,B]$. Then $\alpha(X\cap Y)=\alpha(X)$, and from
Proposition \ref{UQP} immediately follows that the closure space $(E,\alpha)$
is uniquely generated. It is easy to see that the partition to equivalence classes of cospanning relation w.r.t. $\alpha$ coincides with $\mathcal{P^{'}}$.
\end{proof}

The equivalent statement in terms of cospanning relations looks as following.
\begin{theorem}
Let $E$ be a finite set and $R \subseteq 2^{E} \times 2^{E}$ an equivalence relation on $2^{E}$. Then $R$ is the cospanning relation of a convex geometry if and only if the following conditions hold for every $X,Y,Z \subseteq E$ and $x,y \notin X$:

\textbf{R1}: if $(X,Y) \in R$, then $(X,X \cup Y) \in R$

\textbf{R2}: if $X \subseteq Y \subseteq Z$ and $(X,Z) \in R$, then $(X,Y) \in R$.

\textbf{R3}: if $(X,Y) \in R$, then $(X,X \cap Y) \in R$

\textbf{R4}: if $(X,X \cup z) \notin R$ for all $z \notin X$ and $(X,X - x) \notin R$,  then $(X - x,X -x \cup y) \notin R$ for all $y \notin X$.

\end{theorem}

It is easy to see that the property \textbf{R4} is equivalent to accessibility property if a closure space is uniquely generated.  Indeed, \textbf{R3} implies that for each $X \subseteq E$ there is an unique basis $ex(X)$, and so  $(X,X - x) \notin R  \Leftrightarrow  x\in ex(X)$. The second condition says that $X-x$ is a closed set.

\section{Greedoids}

\subsection{Violator spaces and greedoids}
Let we have a greedoid $(E,\mathcal{F})$, i.e.,

(i) $\emptyset \in \mathcal{F}$

(ii) $X,Y \in \mathcal{F}, |X|>|Y| \Rightarrow \exists x \in X-Y. Y \cup x \in \mathcal{F}$

Elements of $\mathcal{F}$ are calls feasible sets.

The rank function of a greedoid is defined as follows:

$r(X)=max\{|A|:A \subseteq X, A \in \mathcal{F}\}$.
Define \textit{(rank) closure operator} (\cite{Greedoids}): $\sigma(X)=\{x: r(X \cup x)= r(X)\}$. If $\Gamma(X)=\{x \in E-X: X \cup x \in \mathcal{F}\}$, then $\sigma(X)=E-\Gamma(X)$.

\begin{lemma} \cite{Greedoids}
\label{L_Gr}
Let $(E,\mathcal{F})$ be a greedoid, then

(i) $X \subseteq \sigma(X)$.

(ii) $(X\subseteq \sigma(Y) \wedge Y\subseteq \sigma(X))\Rightarrow \sigma
(X)=\sigma(Y)$.

(iii) if $x,y \in E - X$, and $X \cup x \in \mathcal{F}$, then $x \in \sigma(X\cup y)\Rightarrow y \in \sigma(X\cup x)$.
\end{lemma}
The last property (iii) may be considered as a weaker version of the Steinitz-MacLine exchange property of matroids:

 if $x \notin \sigma(X)$,  $x \in \sigma(X\cup y)$\ then $y \in \sigma(X\cup x)$.

Since (ii) is equivalent to self-convexity (Proposition \ref{SC_G}) we can see that
 greedoids may be considered as a subclass of violator spaces.

\begin{lemma} \cite{Greedoids}
Let $(E,\mathcal{F})$ be a greedoid, then

$\mathcal{F}=\{X\subseteq E: \forall x \in X: x \notin \sigma(X-x)\}$.
\end{lemma}

So each element of a feasible set is an extreme point of the set, i.e., $\mathcal{F}=\{X\subseteq E: X=ex(X)\}$.

The following definition of feasible sets is equivalent:
\begin{lemma}
\label{BB}
Let $(E,\mathcal{F})$ be a greedoid, then

$\mathcal{F}=\{X\subseteq E: \forall Y \subset X: \sigma(X)\neq \sigma(Y) \}$ - the family of bases -  minimal generators w.r.t. operator $\sigma$.
\end{lemma}

\begin{proof}
If for all $Y \subset X: \sigma(X)\neq \sigma(Y) $, then for all $x \in X$ holds $\sigma(X)\neq\sigma(X-x)$, and so ( from Corollary \ref{Lemma}) $x \notin \sigma(X-x)$ (or $X = ex(X)$). If there exists $Y \subset X$ such that $\sigma(X) = \sigma(Y)$, then for each $x \in X-Y$ convexity implies $\sigma(X) = \sigma(X-x)$.
\end{proof}

\begin{remark}
For a greedoid $(E,\mathcal{F})$ a basis of $X \subseteq E$ is defined as a maximal feasible subset of $X$, while for a space $(E, \sigma)$ a basis is defined as a minimal generator of $X$. Lemma \ref{BB} provides the equivalence of two definitions.
\end{remark}

Consider the third property (iii). If $x,y \in E - X$ and $x \in \sigma(X\cup y)$, then $r(X\cup x \cup y)=r(X\cup y) \leq |X|+1$.  Since $X \cup x \in \mathcal{F}$, we have $|X|+1 \leq r(X\cup x \cup y)=r(X\cup y) \leq |X|+1=r(X\cup x)$. Hence $X \cup y \in \mathcal{F}$. Corollary \ref{Lemma} implies that  $\sigma(X \cup x)=\sigma(X\cup x \cup y)=\sigma(X \cup y)$, and so $X\cup x \cup y$ has two bases $X\cup x$ and $X \cup y$.

The third property (iii) is equivalent to the following property:

(iv): if $\sigma(X \cup y)=\sigma(X\cup x \cup y)$, but $\sigma(X \cup x) \neq \sigma(X\cup x \cup y)$, then $X \cup x \notin \mathcal{F}$, i.e., there exists a $z \in X \cup x$ such that $\sigma(X \cup x - z)=\sigma(X\cup x)$.

\subsection{Cospanning relation of greedoids}

Let $(E,\mathcal{F})$ be a greedoid and $\sigma$ be a closure operator of the greedoid. The two sets $X$ and $Y$ are \textit{equivalent} (or cospanning)
if $\sigma(X)=\sigma(Y)$.  In what follows, $\mathcal{P}$ denotes a partition
of $H(E)$ ($2^{E}$) into equivalence classes with regard to this relation, and
$[A]_{\sigma}=\{X\subseteq E:\sigma(X)=\sigma(A)\}$. The following theorem characterizes the cospanning relation of greedoids.

\begin{theorem}
\cite{Greedoids}
\label{T_Gr}
Let $E$ be a finite set and $R \subseteq 2^{E} \times 2^{E}$ an equivalence relation on $2^{E}$. Then $R$ is the cospanning relation of a greedoid if and only if the following conditions hold for every $X,Y,Z \subseteq E$ and $x,y \notin X$:

\textbf{R1}: if $(X,Y) \in R$, then $(X,X \cup Y) \in R$

\textbf{R2}: if $X \subseteq Y \subseteq Z$ and $(X,Z) \in R$, then $(X,Y) \in R$.

\textbf{R4}: if $(X \cup y,X \cup x \cup y) \in R$, but $(X \cup x,X \cup x \cup y) \notin R$, then there exists an element $z \in X \cup x$ such that $(X \cup x - z,X \cup x) \in R$.
\end{theorem}

 Since we have made some changes in the proof to use the obtained results in the future, we give here all the proof.

\begin{proof}
Necessity follows immediately from Lemma \ref{L_Gr} and  Theorem \ref{T_rel}. To prove the sufficiency, we build
$\mathcal{F}=\{X\subseteq E: \forall Y \subset X: \sigma(X)\neq \sigma(Y) \}$
- the family of minimal sets of equivalence classes, and check that the family is a greedoid.

\begin{lemma}
\label{aug_p}
The property \textbf{R4} is equivalent to the following \textit{augmentation property}:

If $A \in \mathcal{F}$, and  $a \in E-A$ such that
 $(A,A \cup a) \notin R$, then $A \cup a \in \mathcal{F}$.
 \end{lemma}

 In other words, the augmentation property means that if $A$ is a minimal element of some equivalence class (and so it belongs to $\mathcal{F}$), then for each $a\in E-A$ such that $A \cup a$ does not belong to this equivalence class $[A]$, $A \cup a$ is a minimal element of another equivalence class.

 Let us prove that the augmentation property follows from \textbf{R4}. Suppose there exists $a \in E-A$ such that $(A,A \cup a) \notin R$, but $A \cup a \notin \mathcal{F}$. Hence there is some $b\neq a$ such that $(A\cup a - b,A \cup a) \in R$. Set $X=A-b$, then $(X \cup a,X \cup a \cup b) \in R$,  $(X \cup b,X \cup a \cup b) \notin R$, then $A=X \cup b \notin \mathcal{F}$. Contradiction.

 To prove that \textbf{R4} follows from the augmentation property suppose that $X \cup x \in \mathcal{F}$. Denote $A=X \cup x$. Then $(A,A \cup y) \notin R$, and so $A \cup y \in \mathcal{F}$, in contradiction to $(A \cup y - x, A \cup y) \in R$.

\begin{lemma}
\label{claim2}
Let $A \subseteq X \subseteq E$ and $A \in \mathcal{F}$. Then  there exists $x \in X - A$ such that $A \cup x \in \mathcal{F}$ if and only if $(A,X) \notin R$.
\end{lemma}

Indeed, if $(A,X) \in R$, then convexity \textbf{R2} implies that $(A,A \cup x)\in R$ for each $x \in X - A$, and hence $A \cup x \notin \mathcal{F}$. Conversely, if $A \cup x \notin \mathcal{F}$ for all $x \in X - A$, then from augmentation property follows that $(A,A \cup x)\in R$ for all $x \in X - A$, and so by \textbf{R1}, $(A,X) \in R$.

From definition of $\mathcal{F}$ it follows that $\emptyset \in \mathcal{F}$. Thus Lemma \ref{claim2} implies that $\mathcal{F}$ is accessible. To prove that $\mathcal{F}$ is a greedoid we have to check that for each $A,B \in \mathcal{F}$ with $|A|<|B|$ there exists some $b \in B - A$ such that $A \cup b \in \mathcal{F}$. Assume that the property does not hold. Let $C \subseteq A \cap B$, $C \in \mathcal{F}$, and choose $A,B,C$ such that $|C|$ is maximal. By Lemma \ref{claim2} $(A,A \cup B) \in R$, and so $A\cup B \notin \mathcal{F}$ and $A \nsubseteq B$.

$C \subset B$ and $(C,B) \notin R$, then by Lemma \ref{claim2} there exists a $c \in B - C$ such that $C \cup c \in \mathcal{F}$. By the maximality of $C$ , $c \notin A$. From the assumption, $(A,A \cup c) \in R$. Since $C \cup c \subseteq A \cup c$, Lemma \ref{claim2} implies  that there is a chain $C_{1}=C \cup c \subseteq C_{2}=C_{1}\cup x_{1} \subseteq C_{3}=C_{2}\cup x_{2}\subseteq...\subseteq C_{k}=D$,  where all elements of the chain $C_{i} \in \mathcal{F}$ and $C_{i} \subseteq A \cup c$. The chain ends with $D \in \mathcal{F}$ and $(D,A \cup c) \in R$. Then $(A,D)\in R$, and so $(D,A \cup B) \in R$. Hence $D$ cannot be augmenting from $B$. Note, that $D \subset A\cup c$, and so $|D| \leq |A| < |B|$. Since $C \cup c \subseteq D \cap B$, this contradicts the choice of $A,B,C$.

Hence we proved that the family $\mathcal{F}$ of minimal sets of equivalence classes forms a greedoid. One corollary is that all minimal sets in an equivalence class have the same cardinality.

 It remains to show that the cospanning relation of a greedoid $\mathcal{F}$ is $R$. Let $A$ be a maximal feasible set of $X$, then , by Lemma \ref{claim2}, $(A,X) \in R$, and $(A,\sigma(X)) \in R$, so $(X,\sigma(X)) \in R$ for all $X \subseteq E$. \textbf{R1} implies, that each equivalence class has an unique maximal element. Thus for each maximal element $Z$  we have $Z=\sigma(Z)$.

 If $\sigma(X)=\sigma(Y)$, then $(Y,\sigma(X)) \in R$, and hence $(X,Y) \in R$.  Conversely, let $(X,Y) \in R$, then \textbf{R1} implies that $(X,X \cup Y) \in R$. Let $A$ is a maximal feasible set of $X$. Then $(A,X \cup Y) \in R$ and, by Lemma \ref{claim2}, $A$ is a maximal feasible set of $X \cup Y$. So $X \cup Y \subseteq \sigma(A)=\sigma(X)$. So $Y \subseteq \sigma(X)$. Similarly, $X \subseteq \sigma(Y)$, and by Lemma \ref{L_Gr}, $\sigma(X)=\sigma(Y)$. Then for each $X \in [A]$, $\sigma(X)$ equals to the unique maximal element of $[A]$. Thus the relation $R$ coincides with  cospanning relation w.r.t. operator $\sigma$.
\end{proof}

\begin{theorem}
Let we have some space $(E,\alpha)$. Then $\alpha$ is the closure  operator of a greedoid if and only if  $\alpha$ satisfies the property

\textbf{G1}: $X\subseteq \alpha(X)$ (extensivity),

\textbf{G2}: $(X\subseteq Y\subseteq \alpha(X))\Rightarrow \alpha
(X)=\alpha(Y)$ (self-convexity).

\textbf{G3}: if $\alpha(X \cup y)=\alpha(X\cup x \cup y)$, but $\alpha(X \cup x) \neq \alpha(X\cup x \cup y)$, then there exists a $z \in X \cup x$ such that $\alpha(X \cup x - z)=\alpha(X\cup x)$.
\end{theorem}

Hence each violator space $(E,\varphi)$ determines a greedoid if operator $\varphi$ satisfies \textbf{G3}.

\begin{proof}
We have already seen that the properties are necessary. To prove sufficiency define an equivalence relation $R$ w.r.t. $\alpha$. Then $R$ is the cospanning relation of a greedoid $(E,\mathcal{F})$ with the closure operator of the greedoid $\sigma$. Since $X$ and $\sigma(X)$ are cospanning, \textbf{G1} implies $\sigma(X) \subseteq \alpha(\sigma(X))=\alpha(X)$. Since extensivity and self-convexity implies idempotence, $X$ and $\alpha(X)$ are cospanning, so $\alpha(X) \subseteq \sigma(X)$.
\end{proof}

Thus, if we have cospanning relation satisfying the properties \textbf{R1},\textbf{R2} and \textbf{R4} the family $\mathcal{F}$ of minimal sets of equivalence classes forms a greedoid. It is interesting notice, that
the third property \textbf{R4} is not necessary for $\mathcal{F}$ to be a greedoid.
\begin{example}
Let $E=\{1,2,3\}$. Define $\varphi(X)=X$ for each $X\subseteq E$ except
$\varphi(\{1\})=\varphi(\{1,3\})=\{1,3\}$. It is easy to check that the space $(E,\varphi)$ is a
uniquely generated violator space (satisfies both extensivity and self-convexity), where the family of bases $\mathcal{F}=P(\{1,2,3\}) - \{1,3\}$ forms a greedoid. At the same time, the operator $\varphi$ does not satisfies the property \textbf{G3} and the cospanning relation w.r.t. $\varphi$ does not satisfies \textbf{R4}. Indeed,  if $X=\emptyset, x=3,y=1$, then $X \cup x =\{3\} \in \mathcal{F}$, $(X \cup y,X \cup x \cup y) \in R$, and $(X \cup x,X \cup x \cup y) \notin R$.

If we consider the rank function of the greedoid , we can see that $r(\{1\})=r(\{3\})=r(\{1,3\})$, and so
$\sigma(\{1\})=\sigma(\{3\})=\{1,3\}$. Then for this function $\sigma$ the property \textbf{R4} holds and we have a not uniquely generated violator space.
\end{example}

So the same family of bases may be obtain by different operators and by different equivalence relations.

\subsection{Antimatroids - uniquely generated greedoids}

An antimatroid is a greedoid closed under union.

\begin{lemma}\cite{Greedoids}
For a accessible set system $(E,\mathcal{F})$ the following statements are equivalent:

(i)  $(E,\mathcal{F})$ is an antimatroid.

(ii) $\mathcal{F}$ is closed under union.

(iii) $A, A \cup x, A \cup y \in \mathcal{F}$ implies $A \cup \{x,y\} \in \mathcal{F}$

\end{lemma}

\begin{proposition}
An antimatroid  is a uniquely generated greedoid.
\end{proposition}

\begin{proof}
Let $\mathcal{F}$  be an antimatroid. Since each antimatroid is a greedoid, it remains to prove that the greedoid is uniquely generated. Suppose there are two bases $B_{1}$ and $B_{2}$ such that $\sigma(B_{1}) = \sigma(B_{2})$. Since $\mathcal{F}$ is a family of bases, then $B_{1},B_{2} \in \mathcal{F}$. Hence $B_{1} \cup B_{2} \in \mathcal{F}$, because $\mathcal{F}$ is an antimatroid. But $\sigma(B_{1} \cup B_{2}) = \sigma(B_{1})$ (see Lemma \ref{un}). Contradiction.
\end{proof}

Since for each greedoids the operator $\sigma$ is a violator operator, Theorem \ref{TH} implies
\begin{corollary}
The operator $\sigma$ of each antimatroid satisfies the anti-exchange property.
\end{corollary}

Thus we can conclude with the following theorem.

\begin{theorem}
\label{UG_A}
The family $\mathcal{F}$ is an antimatroid if and only if $\mathcal{F}$ is a uniquely generated greedoid.
\end{theorem}

\begin{proof}
It remains to prove that each uniquely generated greedoid is an antimatroid. Suppose,
$A, A \cup x, A \cup y \in \mathcal{F}$, but $A \cup \{x,y\} \notin \mathcal{F}$. Then $x,y \notin \sigma(A)$, $x \in \sigma(A\cup y)$ and $y \in \sigma(A\cup x)$. Contradiction to anti-exchange property.
\end{proof}

\begin{remark}
Antimatroids being greedoids satisfy the weaker version of exchange property (iii) ( Lemma \ref{L_Gr}) and at the same time they satisfy the anti-exchange property. But that is not a problem (paradox), since the conditions of the property (iii) ($X \cup x \in \mathcal{F}$, and $x \in \sigma(X\cup y)$) do not hold for antimatroids.
\end{remark}

Since any antimatroid is an uniquely generated greedoid, the cospanning relation w.r.t. the closure operator of an antimatroid, or, for simplicity, the cospanning relation of an antimatroid, determines the hypercube partition $\mathcal{P}$ with augmentation property (see Lemma \ref{aug_p}):
\begin{equation}
\forall [A,B]\in \mathcal{P} \text{ and } x \notin B \text{ there is }C \subseteq E \text{ such that } [A \cup x,C] \in \mathcal{P}\label{Eq_AN}%
\end{equation}

Let we have some hypercube partition $\mathcal{P^{'}}$ with augmentation property (\ref{Eq_AN}). Denote $\mathcal{F}=\{A:[A,B]\in \mathcal{P^{'}}\}$.

\begin{theorem}
Every hypercube partition $\mathcal{P^{'}}$ satisfying (\ref{Eq_AN}) is the partition  of $H(E)$ into equivalence classes w.r.t. the closure operator of an antimatroid $\mathcal{F}$.
\end{theorem}

The equivalent statement looks as following.
\begin{theorem}
Let $E$ be a finite set and $R \subseteq 2^{E} \times 2^{E}$ an equivalence relation on $2^{E}$. Then $R$ is the cospanning relation of an antimatroid if and only if the following conditions hold for every $X,Y,Z \subseteq E$ and $x,y \notin X$:

\textbf{R1}: if $(X,Y) \in R$, then $(X,X \cup Y) \in R$

\textbf{R2}: if $X \subseteq Y \subseteq Z$ and $(X,Z) \in R$, then $(X,Y) \in R$.

\textbf{R3}: if $(X,Y) \in R$, then $(X,X \cap Y) \in R$

\textbf{R4}: if $(X \cup y,X \cup x \cup y) \in R$, but $(X \cup x,X \cup x \cup y) \notin R$, then there exists an element $z \in X \cup x$ such that $(X \cup x - z,X \cup x) \in R$.

\end{theorem}

Since hypercube partition satisfying (\ref{Eq_AN}) determines the uniquely generated greedoid (Theorem \ref{T_Gr}), the theorem immediately follows from Theorem \ref{UG_A}.

In any way, it may be interesting to see how the uniqueness  of a basis turns a greedoid to be an antimatroid.

\begin{proof}
Necessity follows immediately from Theorems \ref{T_Gr} and  \ref{UG_A}. Since properties \textbf{R1},\textbf{R2},\textbf{R4} imply that the family $\mathcal{F}$ is a greedoid (Theorem \ref{T_Gr}) to prove the sufficiency it remains to prove that $\mathcal{F}$ is closed under union.

Let $A,B \in \mathcal{F}$. $(A,B) \notin R$, since there is an unique basis. Then, w.l.o.g., $(B,A \cup B) \notin R)$. Lemma \ref{claim2} implies that there is a chain $B=B_{0}\subset B_{1}\subset...\subset B_{k}=C$ such that
$B_{i}=B_{i-1}\cup x_{i}$, $B_{i}\in\mathcal{F}$ for $0\leq i\leq k$, and $(C, A\cup B) \in R$.

If $(A, A\cup B) \in R$, then $A=C$ (from uniqueness of the basis), and so $A\cup B = A \supseteq B$.
If $(A, A\cup B) \notin R$, then there exists a chain $A=A_{0}\subset A_{1}\subset...\subset A_{m}=D$,
such that $A_{i}\in\mathcal{F}$ for $0\leq i\leq m$, and $(D, A\cup B) \in R$. Hence, $C=D=A\cup B$. So
$\mathcal{F}$ is closed under union.
\end{proof}

\subsection{Matroids - hereditary greedoids}

A matroid $(E,\mathcal{F})$ may be defined as a greedoid  which satisfy hereditary property:

for each $Y \subseteq X \subseteq E$, if $X \in \mathcal{F}$, then $Y \in \mathcal{F}$.

\begin{theorem}
Let $E$ be a finite set and $R \subseteq 2^{E} \times 2^{E}$ an equivalence relation on $2^{E}$. Then $R$ is the cospanning relation of a matroid if and only if the following conditions hold for every $X,Y,Z \subseteq E$ and $x,y \notin X$:

\textbf{R1}: if $(X,Y) \in R$, then $(X,X \cup Y) \in R$

\textbf{R2}: if $X \subseteq Y \subseteq Z$ and $(X,Z) \in R$, then $(X,Y) \in R$.

\textbf{R4}: if $(X \cup y,X \cup x \cup y) \in R$, but $(X \cup x,X \cup x \cup y) \notin R$, then there exists an element $z \in X \cup x$ such that $(X \cup x - z,X \cup x) \in R$.

\textbf{R5}: if  $(X,X - x) \notin R$ for each $x \in X$,  then  $(X - x - z,X - x) \notin R$ for each $z \in X - x$.

\end{theorem}

\subsection{Antimatroids and convex geometries}

We know that  $(E,\mathcal{F})$ is an antimatroid, if and only if  $(E,\mathcal{N})$ is a convex geometry, where $ \mathcal{N}=\{E-X: X \in \mathcal{F}\}$. What is a connection between their partitions to cospanning equivalence classes?

First find the formula for bases in antimatroids.
\begin{lemma}
Let $(E,\mathcal{F})$ be an antimatroid. Then $B_{X}$ is a basis of $X \subseteq E$ if and only if
\begin{equation}
\label{basis}
B_{X}=
{\displaystyle\bigcup}
\{A\in \mathcal{F}:A\subseteq X\}.
\end{equation}
\end{lemma}

\begin{proof}
Let $B_{X}$ be a basis of $X$, and so $B_{X} \in \mathcal{F}$. Then if $A \in \mathcal{F}$ and $A\subseteq X$, then $A \subseteq B_{X}$. Indeed, $A \cup B_{X} \subseteq X$ and $A \cup B_{X} \in \mathcal{F}$, so $B_{X}=A \cup B_{X}$, that proves $A \subseteq B_{X}$. Then (\ref{basis}) is proved.

If $B$ is defined by (\ref{basis}), then $B \in \mathcal{F}$, $B \subseteq X$ and $ B_{X} \subseteq B$. Then $B=B_{X}$.
\end{proof}

Since the section deals with equivalence classes of antimatroids and convex geometries, to distinguish them from each other we denote them by $[X]_{\sigma}$ and by $[X]_{\tau}$ correspondingly.

It is easy to see that the maximal element in $[\overline{X}]_{\tau}$ is a compliment of the minimal element in $[X]_{\sigma}$, i.e., $\tau(\overline{X})=\overline{B_{X}}$. Indeed, (\ref{basis}) implies
$\overline{B_{X}}=\bigcap\{\overline{A}\in \mathcal{N}:\overline{X} \subseteq \overline{A}\}$ and from (\ref{tau}) it follows that  $\overline{B_{X}}=\tau(\overline{X})$.

Taking into account that in  antimatroids the unique basis of each set is a set of extreme point we obtain $\overline{ex_{\sigma}(X)}=\tau(\overline{X})$.

\begin{lemma}
Let $(E,\mathcal{F})$ be an antimatroid. Then $\overline{\sigma(X)} = ex_{\tau}(\overline{X})$.
\end{lemma}

\begin{proof}
Consider an equivalence class $[X]_{\sigma}$. Lemma \ref{aug_p} implies that for each $a \in E-ex_{\sigma}(X)$ the following holds: $a \notin \sigma(X)$ if and only if $(ex_{\sigma}(X) \cup a) \in \mathcal{F}$. Note, that $ex_{\tau}(\overline{X}) \subseteq \tau(\overline{X})=\overline{ex_{\sigma}(X)}$, and so if $a \in ex_{\tau}(\overline{X})$ then $a \in E-ex_{\sigma}(X)$.
Based on (\ref{Augm_P}) we have

$a \in ex_{\tau}(\overline{X})\Leftrightarrow \tau(\overline{X})- a \in \mathcal{N} \Leftrightarrow \overline{ex_{\sigma}(X)}- a \in \mathcal{N} \Leftrightarrow (ex_{\sigma}(X) \cup a) \in \mathcal{F} \Leftrightarrow a \notin \sigma(X)$.
\end{proof}

So Proposition \ref{co-co} immediately implies
\begin{proposition}
There is a one-to-one correspondence between an equivalence class of $X$ of cospanning relation w.r.t. an antimatroid - $[X]_{\sigma}$ and an equivalence class $[\overline{X}]_{\tau}$ in the convex geometry $\mathcal{N}$, i.e., $A \in [X]_{\sigma}$ if and only if $\overline{A} \in [\overline{X}]_{\tau}$.
 \end{proposition}

 \begin{corollary}
 Let $(E,\mathcal{F})$ be an antimatroid. Then

 $X \subseteq Y \subseteq E \Rightarrow ex_{\sigma}(X) \subseteq ex_{\sigma}(Y)$
 \end{corollary}

 Indeed, $X \subseteq Y \Rightarrow \overline{Y} \subseteq \overline{X} \Rightarrow \tau(\overline{Y}) \subseteq \tau(\overline{X}) \Rightarrow \overline{\tau(\overline{X})} \subseteq \overline{\tau(\overline{Y})} \Rightarrow ex_{\sigma}(X) \subseteq ex_{\sigma}(Y)$.

 \section{Conclusion}

 Many combinatorial structures may be characterized by using operators defined on the elements of the structures. Thus matroids are described by a closure operator with exchange property, convex geometries are usually defined as a closure space with anti-exchange property. Greedoids may be described by some "closure operator" as well. Each set operator determines the partition of sets to equivalence classes with equal value of the operator and thus it generates the cospanning equivalence relation on sets. The cospanning relation associated with a closure operator of greedoids was introduced and investigated in \cite{Greedoids}, where it was proved that the relation determines the greedoid uniquely. We extended this approach to another combinatorial structures and obtain cospanning characterization of violator and co-violator spaces, for convex geometries, antimatroids, and matroids.

 It remains an open problem to characterize the cospanning partition and/or the
cospanning relation of closure spaces.

\end{document}